\documentclass[12pt]{amsart}

\usepackage{amssymb,amsmath,amsfonts,tikz,epsfig,float}
\usepackage{amsthm}
\usepackage{mathrsfs}
\usepackage{enumitem,indentfirst}
\usepackage[fleqn,tbtags]{mathtools}
\usepackage{color}
\usepackage{a4wide}
\usepackage{cleveref}
\usepackage{euscript}
\usepackage{bbm}
\usepackage[v2,tips]{xy}

\newtheorem{Thm}{Theorem}[section]
\newtheorem{Lem}[Thm]{Lemma}
\newtheorem{Prop}[Thm]{Proposition}

\newtheorem{Cor}[Thm]{Corollary}
\theoremstyle{definition}
\newtheorem{Def}[Thm]{Definition}

\newtheorem{Rem}[Thm]{Remark}
\newtheorem{Que}[Thm]{Question}

\newtheorem*{Ack}{Acknowledgements}
\numberwithin{equation}{section}

\renewcommand{\le}{\ensuremath{\leqslant}}

\renewcommand{\ge}{\ensuremath{\geqslant}}

\newcommand{\lhs}{\operatorname{lhs}}
\newcommand{\rhs}{\operatorname{rhs}}
\newcommand{\supp}{\operatorname{supp}}

\newcommand{\one}{\mathbbm{1}}

\newcommand{\A}{\mathcal{A}}

\newcommand{\mc}{\mathcal}

\newcommand{\N}{\mathbb{N}}

\title[Kernels of operators on  Banach spaces]{Kernels of operators on Banach spaces\\ induced by almost disjoint families}

\dedicatory{In memoriam: H.~G.~Dales (1944--2022)}

\author[B.~Horv\'{a}th]{Bence Horv\'{a}th} 
\address{(B.~Horv\'{a}th) Baloise Insurance Ltd, Aeschengraben~21, Postfach 2275, Basel 4002, Switzerland}
\email{hotvath@gmail.com}

\author[N.~J.~Laustsen]{Niels Jakob Laustsen} 
\address{(N.~J.~Laustsen) School of Mathematical Sciences, Fylde
  College, Lancaster University, Lancaster LA1 4YF, United Kingdom}
\email{n.laustsen@lancaster.ac.uk}

\begin{document}	

\begin{abstract} Let~$\mathcal{A}$ be an almost disjoint family of subsets of an infinite set~$\Gamma$, and denote by~$X_{\mathcal{A}}$ the closed subspace of~$\ell_\infty(\Gamma)$ spanned by the indicator functions of intersections of finitely many sets in~$\mathcal{A}$. We show that if~$\mathcal{A}$ has cardinality greater than~$\Gamma$, then the closed subspace of~$X_{\mathcal{A}}$ spanned by the indicator functions of sets of the form $\bigcap_{j=1}^{n+1}A_j$, where $n\in\N$ and $A_1,\ldots,A_{n+1}\in\mathcal{A}$ are distinct,  cannot be the kernel of any bounded operator \mbox{$X_{\mathcal{A}}\rightarrow \ell_{\infty}(\Gamma)$}. As a consequence, we deduce that the subspace   
\[ \bigl\{ x\in \ell_{\infty}(\Gamma) : \text{the set}\  \{\gamma \in \Gamma : \lvert x(\gamma)\rvert > \varepsilon \}\ \text{has cardinality smaller than}\ \Gamma\ \text{for every}\ \varepsilon>0\bigr\} \] 
of~$\ell_\infty(\Gamma)$ is not the kernel of any bounded operator on~$\ell_\infty(\Gamma)$; this generalises results of Kalton and of Pe\l{}czy\'{n}ski and Sudakov.

The situation is more complex for the Banach space~$\ell_\infty^c(\Gamma)$ of countably supported, bounded functions defined on an uncountable set~$\Gamma$. We show that it is undecidable in \textsf{ZFC} whether every bounded operator on~$\ell_\infty^c(\omega_1)$ which vanishes on~$c_0(\omega_1)$ must vanish on a subspace of the form~$\ell_\infty^c(A)$ for some uncountable subset~$A$ of~$\omega_1$.\bigskip

To appear in \emph{Houston Journal of Mathematics.}\vspace{-1cm}
\end{abstract}
\maketitle

\let\thefootnote\relax\footnote{\subjclass{2020}{\textit{~Mathematics Subject Classification.} Primary
    46B26,
    %Nonseparable Banach spaces
    46E15,
    % Banach spaces of continuous, differentiable or analytic functions
47B01; %  	Operators on Banach spaces
    Secondary
    18A30
 %   Limits and colimits (products, sums, directed limits, pushouts, fiber products, equalizers, kernels, ends and coends, etc.)
    47B38.%
% Linear operators on function spaces (general)
  }
  
\keywords{\textit{Key words and phrases.} Non-separable Banach space, bounded operator, kernel, almost disjoint family.}}

\section{Introduction and statement of the main result}\label{S:intro}
\noindent
The problem of classifying all complemented subspaces --- that is, the closed subspaces which arise as the kernels of bounded idempotent operators --- of a given Banach space is fundamental in the study of Banach spaces. This problem has a very natural generalisation, obtained by omitting the word ``idempotent'',   originally formulated in \cite{LaW}:
\begin{Que}\label{Q:LW}
  Which closed subspaces of a Banach space~$X$ can be realised as the kernel of a bounded operator $X \rightarrow X$?
  In particular, for which Banach spaces $X$ is it true that every closed subspace of $X$ is the kernel of some bounded operator $X\to X$?
\end{Que}

In the positive direction, White and the second author \cite[Proposition~2.1]{LaW} observed that whenever~$W$ is a closed subspace of a Banach space $X$ such that~$X/W$ is separable, there is a bounded operator $X \rightarrow X$ whose kernel is~$W$. Consequently, \Cref{Q:LW} is of interest only for  non-separable Banach spaces~$X$.

The focus in~\cite{LaW} was on reflexive spaces; its main result \cite[Corollary~2.8]{LaW} states that the dual space~$X^*$ of Wark's non-separable reflexive Banach space~$X$ with ``few operators'' introduced in~\cite{Wark} contains a closed subspace which is not the kernel of any bounded operator $X^*\to X^*$.  More recently, Arnott and the second author~\cite{AL} have shown that for $1<p<\infty$ and an arbitrary index set~$\Gamma$, every closed subspace of $X=\ell_p(\Gamma)$ and  $X=c_0(\Gamma)$ is the kernel of a bounded operator $X\to X$, whereas~$\ell_1(\Gamma)$ contains closed sub\-spaces which are not the kernel of any bounded operator $\ell_1(\Gamma)\to\ell_1(\Gamma)$ whenever~$\Gamma$ is un\-countable. 

Many years prior to this work,  Kalton \cite[Proposition~4]{Ka} showed that~$c_0$ is not the kernel of any bounded operator $\ell_\infty\to\ell_\infty$, thereby
generalising Phillips' Theorem that~$c_0$ is not complemented in~$\ell_\infty$. More precisely, building on the approach Whitley~\cite{Whit} used in his simplified proof of Phillips' Theorem, Kalton proved that whenever the kernel of a bounded operator $T\colon\ell_\infty\to\ell_\infty$ contains~$c_0$, there is an infinite subset~$M$ of~$\N$ such that $\ell_\infty(M)\subseteq \ker T$, where we have identified~$\ell_\infty(M)$ with the subspace \mbox{$\{x\in\ell_\infty : x(n)=0\ (n\in\N\setminus M)\}$} of~$\ell_\infty$. 

Our main theorem generalises Kalton's result in two ways: first, it replaces the index set~$\N$ with an arbitrary infinite set~$\Gamma$, and second, it applies to operators defined on certain $C(K)$\nobreakdash-sub\-spaces of~$\ell_\infty(\Gamma)$ that have found a number of significant applications recently. Before we can state it precisely, we must introduce some notation and terminology.

Since our results depend only on the cardinality of the index set~$\Gamma$, throughout this paper $\Gamma$~will denote an infinite cardinal number unless otherwise stated, and we use the notation
\begin{align*}
[\Gamma]^\Gamma &= \{ A\subseteq\Gamma : \lvert A\vert = \Gamma \}\qquad\text{and}\qquad 
[\Gamma]^{<\Gamma} = \{ A\subseteq\Gamma : \lvert A\rvert < \Gamma \},
\end{align*}
where~$\lvert A\rvert$ denotes the cardinality of the set~$A$. 
A family $\mc{A}\subset[\Gamma]^\Gamma$ is called \emph{almost disjoint} if $A\cap B\in [\Gamma]^{<\Gamma}$ whenever $A,B\in\mc{A}$ are distinct. 
A famous result of Sierpi\'{n}ski~\cite{Sier} states that~$[\Gamma]^\Gamma$ contains  an almost disjoint family of cardinality greater than~$\Gamma$.

We write $\one_A$ for the indicator function of a subset~$A$ of~$\Gamma$, considered as an element of~$\ell_\infty(\Gamma)$. The following closed subspace of~$\ell_\infty(\Gamma)$ induced by a family~$\mc{A}$ of subsets of~$\Gamma$ is at the heart of our work: 
\begin{equation}\label{YAdefn} X_{\mc{A}} = \overline{\operatorname{span}}\, \{\one_{\bigcap_{j=1}^n A_j} : n\in\N,\,A_1,\ldots,A_n\in\mc{A}\} = \overline{\operatorname{span}}\, \{\one_A : A\in\mc{A}\cup\A_\doublecap\},
\end{equation}  
where we have introduced the symbol 
\begin{equation}\label{Eq:Adoublecap} \A_\doublecap = \biggl\{\bigcap_{j=1}^{n+1} A_j : n\in\N,\,A_1,\ldots,A_{n+1}\in\mc{A},\, A_j\ne A_k\ \text{for}\ j\ne k\biggr\} \end{equation}
for the collection of finite intersections of at least two distinct sets in~$\A$. With this notation at hand, we can state our main result as follows.
\begin{Thm}\label{mainthm} Let $\mc{A}\subset[\Gamma]^\Gamma$ be an almost disjoint family of cardinality greater than~$\Gamma$, and suppose that $T\colon X_{\mc{A}}\to\ell_\infty(\Gamma)$ is a bounded operator for which $T\one_A=0$ for every $A\in\A_\doublecap$. Then~$\A$ contains a subset~$\mc{B}$ for which $\lvert\mc{A}\setminus\mc{B}\rvert\leqslant \Gamma$ and   $T\one_B=0$ for every  $B\in\mc{B}$.
\end{Thm}

Note that $\one_B\notin\overline{\operatorname{span}}\, \{\one_A : A\in\A_\doublecap\}$ for every $B\in\A$  (see \Cref{L:suppW0} below for details), so \Cref{mainthm} implies that $\overline{\operatorname{span}}\, \{\one_A : A\in\A_\doublecap\}$ is a  closed subspace of~$X_\A$ which cannot  be the kernel of any bounded operator $X_{\mc{A}}\to\ell_\infty(\Gamma)$.
This answers the second part of \Cref{Q:LW} in the negative for both  $X = X_{\mc{A}}$ (by 
  applying  \Cref{mainthm} to the composition of
 an operator $X_{\mc{A}}\to X_{\mc{A}}$ with the inclusion map $X_{\mc{A}}\to\ell_\infty(\Gamma)$) and 
 $X=\ell_\infty(\Gamma)$ (by applying  \Cref{mainthm} to the restriction to~$X_{\mc{A}}$ of an operator $\ell_\infty(\Gamma)\to\ell_\infty(\Gamma)$).
 
 In fact, with a small amount of extra effort, we can prove a variant of the latter conclusion that  will simultaneously generalise the previously mentioned result of Kalton \cite[Proposition~4]{Ka} to uncountable index sets and a theorem of Pe\l{}czy\'{n}ski and Sudakov \cite[Theorem~1]{PS} stating that the closed subspace
\begin{align}\label{eq:subspace}
\ell_{\infty}^{<}(\Gamma) &= \bigl\{ x\in \ell_{\infty}(\Gamma) : \{\gamma \in \Gamma : \lvert x(\gamma)\rvert > \varepsilon \}\in[\Gamma]^{<\Gamma}\ \text{for every}\ \varepsilon>0\bigr\}
\end{align}
is not complemented in~$\ell_\infty(\Gamma)$ for any infinite cardinal number~$\Gamma$. 

\begin{Cor}\label{corellinty} Let $T\colon\ell_\infty(\Gamma)\to\ell_\infty(\Gamma)$ be a bounded operator whose kernel contains~$\ell_{\infty}^{<}(\Gamma)$. 
Then there exists an almost disjoint family~$\mc{B}\subset [\Gamma]^\Gamma$ of cardinality greater than~$\Gamma$ such that
$X_{\mc{B}}\subseteq\ker T$. 
In particular, $\ell_{\infty}^{<}(\Gamma)$ is a closed subspace of~$\ell_\infty(\Gamma)$ which is not the kernel of any bounded operator $\ell_\infty(\Gamma)\to\ell_\infty(\Gamma)$.
\end{Cor}  

\begin{Rem}  We observe that $\ell_{\infty}^{<}(\N) = c_0$, so \Cref{corellinty} is a genuine generalisation of Kalton's result, which corresponds to $\Gamma=\N$.
\end{Rem}

We shall prove \Cref{mainthm} and \Cref{corellinty} in \Cref{S:proofmain}, using an elementary combinatorial identity~\eqref{CombLemma:eq2} as our key tool. In \Cref{S:CountablySupp} we explore how our results might carry over to the Banach space~$\ell_\infty^c(\Gamma)$ of countably supported, bounded functions defined on an uncountable set~$\Gamma$. Our most interesting conclusion is that the following statement is independent of \textsf{ZFC}:
\begin{equation}\label{ZFC:indep}
\begin{array}{l} \textit{Let}\ T\colon \ell_\infty^c(\omega_1)\to\ell_\infty^c(\omega_1)\  \textit{be a bounded operator whose kernel contains}\  c_0(\omega_1).\\ \textit{Then}\ \ell_{\infty}^c(A)\subseteq\ker T\ \textit{for some uncountable subset}\ A\ \textit{of}\ \omega_1.\end{array}
\end{equation}

\begin{Rem}\label{R:Jan2024}
\begin{enumerate}[label={\normalfont{(\roman*)}}]
\item\label{R:Jan2024:i}  Johnson and Lindenstrauss~\cite{JL} initiated the study of Banach spaces induced by an almost disjoint family~$\mc{A}$ of infinite subsets of~$\N$. In this case it is conventional to consider a slightly larger space than our space~$X_\A$, name\-ly $\overline{X_\A+c_0}$. Clearly, the two spaces are equal if and only if every singleton~$\{n\}$, for $n\in\N$,  arises as the intersection of finitely many sets in~$\A$. Otherwise, by treating the atoms of~$\A_\doublecap$ as single points, one can replace~$\A$ with an almost disjoint family for which equality holds. Therefore, we shall not distinguish between the two approaches and focus on the space~$X_\A$ defined by~\eqref{YAdefn}.
\item\label{R:Jan2024:ii} In the general case  where $\mc{A}\subset [\Gamma]^\Gamma$ is an almost disjoint family for an infinite cardinal number~$\Gamma$ (possibly uncountable),  the Banach space~$X_\A$  has a standard representation as a $C(K)$-space.

The easiest way to see this is to begin with the case of complex scalars: then~$\ell_\infty(\Gamma)$ is a commutative $C^*$-algebra with respect to the pointwise operations, and~$X_\A$ is a closed, self-adjoint subalgebra of it. Therefore the commutative Gel\-fand--Nai\-mark Theorem implies that~$X_\A$ is isometrically $*$-isomorphic to~$C_0(K_\A)$ for some locally compact Hausdorff space~$K_\A$, which  can be (abstractly) described as either the maximal ideal space of the algebra~$X_\A$ or the Stone space of the Boolean subring generated by~$\A$ within the power set of~$\Gamma$.
  
The analogous conclusion for real scalars follows by considering the self-adjoint parts of the above $C^*$-algebras because a real-valued function in either~$\ell_\infty(\Gamma)$, $X_\A$ or~$C_0(K_\A)$ is simply a self-adjoint element of the corresponding complex algebra. 

Combining the isomorphism $X_\A\cong C_0(K_\A)$  with \Cref{mainthm}, we conclude that~$C_0(K_\A)$ contains a closed subspace which is not the kernel of any bounded operator $C_0(K_\A)\to\ell_\infty(\Gamma)$, thereby answering the second part of \Cref{Q:LW} in the negative for $X=C_0(K_\A)$. 
\item For $\Gamma=\N$, the topological space~$K_\A$ has a long and illustrious history going back  nearly a century to the work of Alexandroff and Urysohn~\cite{alexandroff}. It has since appeared under many names, including \emph{Mr\'{o}wka space,} \emph{$\Psi$-space,} \emph{AU-compactum} and \emph{Isbell--Mr\'{o}wka space}. One reason the case $\Gamma=\N$ is attractive is that~$K_\A$ has a nice concrete de\-scrip\-tion: it contains two kinds of points, labelled~$x_n$ for~$n\in\N$ and $y_A$ for $A\in\A$, with the former being isolated, while a neighbourhood basis for the latter consists of all sets of the form $\{x_n:n\in A\setminus F\}\cup\{y_A\}$, where $F\subset A$ is finite.
\end{enumerate}  
\end{Rem}

\section{The proofs of  \Cref{mainthm} and \Cref{corellinty}}\label{S:proofmain}

\noindent Let $\mathbb{K}=\mathbb{R}$ or $\mathbb{K}=\mathbb{C}$ be the scalar field. As usual, the \emph{support} of an element $x\in\ell_\infty(\Gamma)$ is $\supp x = \{ \gamma\in\Gamma : x(\gamma)\ne 0\}$. For an almost disjoint family $\mc{A}\subset[\Gamma]^\Gamma$, define
\begin{equation}\label{mainthm:eq1} V_\A = \operatorname{span} \{ \one_A : A\in\A_\doublecap\}\qquad\text{and}\qquad W_\A = \overline{V_\A}.
\end{equation}
We shall use the following simple fact several times, so we state it formally for ease of reference.

\begin{Lem}\label{L:suppW0}
  Let $\mc{A}\subset[\Gamma]^\Gamma$ be an almost disjoint family.  Then
  \begin{equation}\label{L:suppW0:eq1} \supp v\in[\Gamma]^{<\Gamma}\qquad (v\in V_\A).
\end{equation}    
  Suppose that $\lvert\mc{A}\rvert>\Gamma$. Then $W_{\mc{A}}\ne X_{\mc{B}}$ for every subset $\mc{B}$ of~$\mc{A}$.
\end{Lem}

\begin{proof}    
  The almost disjointness of~$\mc{A}$ implies that $\A_\doublecap\subseteq[\Gamma]^{<\Gamma}$. This proves~\eqref{L:suppW0:eq1} because every element of~$V_\A$ vanishes outside a finite union of sets in~$\A_\doublecap$.

  Suppose that $\lvert\mc{A}\rvert>\Gamma$. Then~$\mc{A}$ cannot consist of disjoint sets, so $W_{\mc{A}}\ne\{0\} = X_\emptyset$. Now consider a non-empty subset~$\mc{B}$ of~$\mc{A}$,   and take  $B\in\mc{B}$. The indicator func\-tion~$\one_B$ belongs to~$X_{\mc{B}}$ by definition. We claim that $\one_B\notin W_{\mc{A}}$, and therefore $X_{\mc{B}}\not\subseteq W_{\mc{A}}$. More precisely, we shall show that~$\one_B$   has distance~$1$ to the dense subspace~$V_\A$ of~$W_{\mc{A}}$. Indeed, the first part of the proof shows that $\supp v\in[\Gamma]^{<\Gamma}$ for	every $v\in V_\A$, whereas $\lvert B\rvert =\Gamma$, so~$B$ contains an element~$\beta$ such that $v(\beta)=0$. Consequently $\lVert \one_B - v\rVert_\infty\geqslant \lvert (\one_B - v)(\beta)\rvert  = 1$, which proves the claim.
\end{proof}

Our next lemma can be viewed as a generalisation of the fact that $c_0(\mathbb{R})$ embeds iso\-met\-ri\-cal\-ly in\-to~$\ell_\infty/c_0$ (see, \emph{e.g.,} \cite[Theorem~1.25]{Avilesetal}).  Before we state it, recall from \Cref{R:Jan2024}\ref{R:Jan2024:ii} that~$X_{\mc{A}}$ is a closed subalgebra of~$\ell_{\infty}(\Gamma)$. Since~$W_{\A}$ is a closed ideal of~$X_{\A}$, the quotient $X_{\A}/W_{\A}$ is a commutative Banach algebra (in fact a $C^*$-algebra in the complex case).
We write  $(e_A)_{A\in\mathcal{A}}$ for the unit vector basis of~$c_0(\mathcal{A})$; that is,
$e_A = \one_{\{A\}}\in\ell_\infty(\mathcal{A})$ for $A\in\mathcal{A}$. We use this notation to ensure a clear distinction between the elements~$e_A\in c_0(\A)$ and~$\one_A\in X_{\mathcal{A}}$.

\begin{Lem}\label{L:c0A} Let $\mc{A}\subset[\Gamma]^\Gamma$ be an almost disjoint family. Then the map given by
\begin{equation}\label{L:c0A:eq0}
  \psi(e_A) = \one_{A} + W_{\mc{A}}\qquad (A \in\mc{A})
\end{equation}
extends uniquely to an isometric algebra isomorphism $\psi \colon c_0(\mc{A}) \to X_{\mc{A}} / W_{\mc{A}}$.
\end{Lem}

In the proof we require the following elementary combinatorial lemma, which may be thought of as a weighted version of the Inclusion--Exclusion Principle, as it is stated in \cite[Proposition~6.62]{MT}, for example. Results of this kind are well-known in the literature, but since we have not been able to find this particular variant anywhere or deduce it from known results, we include a short, self-contained proof of it in \Cref{app:A}.
 
\begin{Lem}\label{CombLemma}
  Let $A_1,\ldots,A_m$ be distinct subsets of~$\Gamma$ for some $m\in\N$, set
  \begin{align}\label{CombLemma:eq1} 
  	A^d_j = A_j\setminus\underset{k\ne j}{\bigcup_{k=1}^m}A_k\qquad (j\in\{1,\ldots,m\}),  
  \end{align}
 and take $\sigma_1,\ldots,\sigma_m\in\mathbb{K}$. Then
  \begin{align}\label{CombLemma:eq2} 
  \sum_{j=1}^m\sigma_j \one_{A^d_j} = \sum_{\emptyset\ne N\subseteq\{1,\ldots,m\}} (-1)^{\lvert N\rvert+1}\Bigl(\sum_{j\in N}\sigma_j\Bigr)\one_{\bigcap_{j\in N} A_j}.
  \end{align}
\end{Lem}

\begin{proof}[Proof of Lemma~{\normalfont{\ref{L:c0A}}}] 
We can define a linear map $\psi\colon\operatorname{span}\{ e_A : A\in \mc{A}\}\to X_{\mc{A}}/W_{\mc{A}}$ by~\eqref{L:c0A:eq0}. Since its domain $\operatorname{span}\{ e_A : A\in \mc{A}\}$  is a dense subalgebra of~$c_0(\mc{A})$, it will suffice to show that this map is multiplicative, isometric and has dense range. 

The multi\-pli\-ca\-tivity of~$\psi$ follows from the fact that  $\one_A\cdot\one_B = \one_{A \cap B}\in V_\A$ whenever $A,B \in \A$ are  distinct.  By definition, the image of~$\psi$ is
\[ \operatorname{span}\{ \psi(e_A) : A\in \mc{A}\} = \operatorname{span}\{ \one_A + W_{\mc{A}} : A\in \mc{A}\}, \]
and this subspace is dense in~$X_{\mc{A}}/W_{\mc{A}}$ because 
 the definitions~\eqref{YAdefn} and~\eqref{mainthm:eq1} of~$X_{\mc{A}}$ and~$V_\A$ imply that the subspace spanned by~$\{ \one_A : A \in \mc{A}\}\cup V_\A$ is dense in~$X_{\mc{A}}$.

 It remains to show that~$\psi$ is an isometry. To this end, take distinct sets $A_1,\ldots,A_m\in\mc{A}$ for some $m\in\N$, define $A_1^d,\ldots,A_m^d$ by~\eqref{CombLemma:eq1},  and let $\sigma_1,\ldots,\sigma_m\in\mathbb{K}$. Then we have
 \[ \psi\Bigl(\sum_{j=1}^m \sigma_j e_{A_j}\Bigr) =  \sum_{j=1}^m \sigma_j \one_{A_j} + W_{\mc{A}} = \sum_{j=1}^m \sigma_j \one_{A^d_j} + W_{\mc{A}}  \]
because \Cref{CombLemma} implies that
\begin{align*}
	\sum_{j=1}^m \sigma_j \one_{A_j} - \sum_{j=1}^m \sigma_j \one_{A^d_j} = \underset{\lvert N\rvert\geqslant 2}{\sum_{N\subseteq\{1,\ldots,m\}}}(-1)^{\lvert N\rvert}\Bigl(\sum_{j\in N}\sigma_j\Bigr)\one_{\bigcap_{j\in N}A_j} \in V_\A \subseteq W_{\mc{A}}.
\end{align*}
Therefore, showing that~$\psi$ is an isometry is equivalent to showing that
\begin{equation}\label{L:c0A:eq} \Bigl\|\sum_{j=1}^m \sigma_j \one_{A^d_j} + W_{\mc{A}}\Bigr\| = \max_{1\le j\le m}\lvert\sigma_j\rvert. 
\end{equation}
The inequality~$\le$ is clear because  the sets $A_1^d,\ldots,A_m^d$ are disjoint. 

To verify the opposite inequality, 
choose $i\in\{1,\ldots,m\}$ such that $\lvert\sigma_{i}\rvert = \max_{1\le j\le m}\lvert\sigma_j\rvert$. For each  $v\in V_\A$, \Cref{L:suppW0} shows that $\supp v\in [\Gamma]^{<\Gamma}$. Consequently, since  $\lvert A_i\rvert=\Gamma$  and $\bigcup_{k=1,\,k\ne i}^m A_k\cap A_i\in[\Gamma]^{<\Gamma}$, we can find  $\gamma\in A_{i}^d$ such that $v(\gamma) =0$. By disjointness, we have $\gamma \notin A_j^d$ for each $j \in \{1, \ldots, m\}\setminus\{i\}$, and therefore
\begin{align*}
	\Bigl\|\sum_{j=1}^m \sigma_j \one_{A_j^d} - v\Bigr\|_{\infty} &\geqslant \Bigl|\sum_{j=1}^m \sigma_j \one_{A_j^d}(\gamma) - v(\gamma) \Bigr| = \lvert\sigma_{i}\rvert.
\end{align*}
This implies that $\bigl\|\sum_{j=1}^m \sigma_j \one_{A^d_j} + W_{\mc{A}}\bigr\|\ge \lvert\sigma_{i}\rvert$ because~$v\in V_\A$ was arbitrary and~$V_\A$ is dense in~$W_{\mc{A}}$; hence~\eqref{L:c0A:eq} follows.
\end{proof}
\begin{Rem} In the case of complex scalars, the map~$\psi$ given by~\eqref{L:c0A:eq0} is a $*$\nobreakdash-iso\-mor\-phism.
\end{Rem}

\begin{Lem}\label{L:zero_set}
Let $U\colon E\to X^*$ be a bounded operator, where $E = c_0(J)$ or $E=\ell_p(J)$ for some infinite set~$J$ and some $p\in (1, \infty)$, and~$X$ is a non-zero Banach space. Then the set $\{j \in J : U e_j \neq 0\}$ has  cardinality no greater than the density character of~$X$, where $(e_j)_{j\in J}$ denotes the unit vector basis for~$E$.
\end{Lem}

\begin{proof} 
  We begin by showing that the set $J_{n, x} = \{j \in J : \lvert\langle U e_j, x\rangle\rvert\geqslant 1/n \}$ is finite for each $n \in \N$ and $x \in X$, where     $\langle\,\cdot\,,\,\cdot\,\rangle$ denotes the duality bracket between~$X$ and its dual space~$X^*$. Assume the contrary, and take $n \in \N$ and $x \in X$ such that~$J_{n, x}$ contains an infinite sequence $(j_k)_{k \in \N}$ of distinct elements.  For each $k \in \N$, choose a scalar~$\sigma_k$ of modulus one such that $\sigma_k \langle U e_{j_k},x\rangle\geqslant 1/n$. Then 
  $y = \sum_{k=1}^{\infty} ({\sigma_k}/{k}) e_{j_k}$ defines an element of~$E$ (recall that we
  do not allow $E = \ell_1(J)$). However, we have
\begin{align*}
  \lvert\langle Uy,x\rangle\rvert  &= \Bigl|\sum_{k=1}^{\infty} \frac{\sigma_k}{k} \langle U e_{j_k},x\rangle\Bigr| \geqslant \frac{1}{n} \sum_{k=1}^{\infty} \frac{1}{k} = \infty,
\end{align*}
which is absurd. This contradiction proves that~$J_{n,x}$ is finite for each $n\in\N$ and \mbox{$x\in X$}. It follows that the set $J_x = \bigcup_{n \in \N} J_{n, x} = \{j \in J : \langle U e_j, x\rangle\ne 0\}$ is countable for each~\mbox{$x\in X$.}

Now take a dense subset~$D$ of $X$. Then $\bigcup_{x\in D} J_{x} = \{j \in J : U e_j \neq 0\}$, and therefore $\lvert\{j \in J :  U e_j \neq 0\}\rvert\leqslant \lvert D\rvert$ because~$D$ is infinite (otherwise it could not be dense in a non-zero Banach space) and the sets~$J_x$ are countable.
\end{proof}

\begin{proof}[Proof of Theorem~{\normalfont{\ref{mainthm}}}.] Suppose that  $T\colon X_{\mc{A}}\to\ell_\infty(\Gamma)$ is a bounded operator for which $T\one_A=0$ for every $A\in\A_\doublecap$,  and define $\mathcal{B} = \{A\in \mathcal{A} : T \one_{A} = 0\}$. Then trivially $T\one_B=0$ for every $B\in\mathcal{B}$, so the conclusion will follow provided that we can show that $\lvert \mc{A}\setminus\mc{B}\rvert\le \Gamma$.

  Since $W_\A\subseteq\ker T$ by hypothesis,  
the Fundamental Isomorphism Theorem implies that we can define a bounded operator $\widetilde{T}\colon X_{\mc{A}}/ W_{\mc{A}}\to \ell_{\infty}(\Gamma)$ by $\widetilde{T}(x + W_{\mc{A}}) = Tx$.
Com\-posing it with  the isometric isomorphism $\psi \colon c_0(\mc{A}) \to X_{\mc{A}} / W_{\mc{A}}$ from \Cref{L:c0A}, we obtain an opera\-tor $U = \widetilde{T}\circ\psi\colon c_0(\A) \to \ell_{\infty}(\Gamma)$ which satisfies $ Ue_{A} =  T \one_{A}$ for every $A\in\mathcal{A}$, and there\-fore 
\begin{equation*}
  \{A \in \mc{A} :  Ue_{A} \neq 0\} = \{A \in \mc{A} : T \one_{A} \neq 0\} = \mc{A}\setminus\mc{B}.
\end{equation*}
\Cref{L:zero_set} shows that  the set on the left-hand side of this equation
  has cardinality at most~$\Gamma$ because  $\ell_\infty(\Gamma)\cong\ell_1(\Gamma)^*$ and~$\ell_1(\Gamma)$ has density character~$\Gamma$. The result follows.
\end{proof}

\begin{proof}[Proof of Corollary~{\normalfont{\ref{corellinty}}}]
  As previously mentioned, Sierpi\'{n}ski~\cite{Sier} has shown that~$[\Gamma]^\Gamma$ contains an almost disjoint family $\mc{A}$ of cardinality greater than~$\Gamma$. The almost disjointness of~$\A$ means that $\one_A\in\ell_\infty^<(\Gamma)$ for every $A\in\A_\doublecap$. Hence, given a bounded operator $T\colon\ell_\infty(\Gamma)\to\ell_\infty(\Gamma)$ with $\ell_{\infty}^{<}(\Gamma)\subseteq\ker T$, we can apply \Cref{mainthm} to the restriction of~$T$ to~$X_{\mc{A}}$ to obtain  a subset~$\mc{B}$
  of~$\mc{A}$ such that $\lvert\mc{A}\setminus\mc{B}\rvert\le\Gamma$ and $X_{\mc{B}}\subseteq\ker T$. We must have $\lvert \mc{B}\rvert>\Gamma$ because $\Gamma<\lvert\mc{A}\rvert = \max\{\lvert \mc{B}\rvert, \lvert \mc{A}\setminus\mc{B}\rvert\}$. This establishes the first part of the result. The second part follows because $\one_B\notin\ell_{\infty}^{<}(\Gamma)$ for every $B\in[\Gamma]^\Gamma$.
\end{proof} 

We conclude this section with a question that we are grateful to the referee for drawing our attention to. For any filter~$\EuScript{F}$ on~$\Gamma$, the set 
\begin{align*}
	c_{\EuScript{F}}(\Gamma) = \bigl\{ x \in \ell_{\infty}(\Gamma) : \lim\limits_{\gamma \to \EuScript{F}} x(\gamma) = 0 \bigr\}
\end{align*}
defines a closed subspace of~$\ell_\infty(\Gamma)$. In the case where $\EuScript{F} = \{ A\subseteq\Gamma : \Gamma \setminus A \in [\Gamma]^{<\Gamma}\}$, 
we have $c_{\EuScript{F}}(\Gamma) = \ell_{\infty}^{<}(\Gamma)$. In view of Corollary~\ref{corellinty}, this raises the following question:

\begin{Que}
	For which filters~$\EuScript{F}$ on~$\Gamma$ is it possible to realise the subspace $c_{\EuScript{F}}(\Gamma)$ of~$\ell_\infty(\Gamma)$ as the kernel of a bounded operator $\ell_\infty(\Gamma)\to\ell_\infty(\Gamma)$?
\end{Que}

We note that  $c_{\EuScript{F}}(\Gamma) = c_0(\Gamma)$ if~$\EuScript{F}$ is the Fr\'{e}chet filter consisting of all cofinite subsets of~$\Gamma$. Let us also remark that in the case where $\EuScript{F} = \{ A\subseteq\Gamma : \alpha\in A\}$ is  the principal ultra\-filter determined by some~$\alpha\in\Gamma$, $c_{\EuScript{F}}(\Gamma)$ is the kernel of a bounded operator, namely the rank-one operator $x\otimes\delta_\alpha\colon \ell_\infty(\Gamma)\to\ell_\infty(\Gamma)$ given by $y\mapsto  y(\alpha)x$, where~$x$ can be any non-zero element of~$\ell_\infty(\Gamma)$.

\section{Kernels of operators on $\ell_{\infty}^{c}(\Gamma)$}\label{S:CountablySupp}
\noindent
A natural variant of the Banach space~$\ell_{\infty}(\Gamma)$ and its subspace~$\ell_{\infty}^{<}(\Gamma)$ defined by~\eqref{eq:subspace} is 
\begin{align*}
	\ell_{\infty}^c(\Gamma) = \{x \in \ell_{\infty}(\Gamma) : \supp x\ \text{is countable}\}.
\end{align*}
This is readily seen to be a closed subspace of~$\ell_{\infty}(\Gamma)$ and hence a Banach space in its own right. Since $\ell_\infty^c(\Gamma)=\ell_\infty(\Gamma)$ when~$\Gamma$ is countable, one can view~$\ell_\infty^c(\Gamma)$ as a ``smaller'' generalisation of~$\ell_\infty$ to uncountable index sets~$\Gamma$ than~$\ell_\infty(\Gamma)$ itself.

From a certain perspective, this  generalisation is nicer than~$\ell_{\infty}(\Gamma)$. Indeed, Johnson, Kania and Schecht\-man \cite[Theorem~1.4]{JKSch} have given a complete classification of the com\-ple\-mented subspaces of~$\ell_\infty^c(\Gamma)$, something that has been achieved for~$\ell_\infty(\Gamma)$ only when~$\Gamma$ is countable. 
However, $\ell_{\infty}^c(\Gamma)$ does not enjoy the same universal properties as $\ell_{\infty}(\Gamma)$ when~$\Gamma$ is uncountable because~$\ell_{\infty}(\Gamma)$ is  injective, whereas~$\ell_\infty^c(\Gamma)$ is not due to a result of  Pe\l{}\-czy\'{n}\-ski and Su\-da\-kov \cite[Corollary, page~87]{PS}, which states that~$\ell_\infty^c(\Gamma)$ is not complemented in~$\ell_\infty(\Gamma)$. We remark that~$\ell_\infty^c(\Gamma)$ has the weaker property of being separably injective (see \cite[Definition~2.1~and~Example~2.4]{Avilesetal} for details).\smallskip

This section is motivated by the following question, which asks whether a natural generalisation of Kalton's result for~$\ell_\infty$ holds true for~$\ell_\infty^c(\Gamma)$.

\begin{Que}\label{question2}  Let $\Gamma$ be an uncountable cardinal number, and 
suppose that  $T\colon \ell_\infty^c(\Gamma)\to\ell_\infty^c(\Gamma)$ is a bounded operator with $c_0(\Gamma)\subseteq \ker T$. Does $\ker T$ contain the subspace 
\[ \ell_{\infty}^c(A) = \{ x\in\ell_{\infty}^c(\Gamma) : \supp x\subseteq A\} \]
for some $A \in [\Gamma]^{\Gamma}$?
\end{Que}

This question  turns out to be undecidable in \textsf{ZFC} for $\Gamma=\omega_1$, as we already stated in \Cref{S:intro}. In one direction, this depends on a deep result of Baum\-gartner \cite[state\-ment~(1) on page~403]{Baum}, who showed that the following statement is independent of  \textsf{ZFC} and the negation of the Continuum Hypothesis:
\begin{equation}\label{eq:Baum}
  \begin{array}{l} \text{There exists a family}\ \mathcal{A} \subset [\omega_1]^{\omega_1}\ \text{of cardinality}\ \omega_2\ \text{such}\\ \text{that}\  A \cap B\ \text{is finite whenever}\ A, B \in \A\ \text{are distinct.}
\end{array}    
  \end{equation}
  
\begin{Thm}\label{t:gch_kernel} Assume Baumgartner's statement~\eqref{eq:Baum}, and let $T \colon \ell_{\infty}^c(\omega_1) \to \ell_{\infty}(\omega_1)$ be a bounded operator whose kernel contains $c_0(\omega_1)$. Then $\ell_{\infty}^c(A)\subseteq \ker T$ for some $A \in [\omega_1]^{\omega_1}$. 
\end{Thm}

\begin{proof}
  Take a family $\mathcal{A}\subset [\omega_1]^{\omega_1}$ of cardinality~$\omega_2$ such that $A \cap B$ is finite whenever \mbox{$A, B \in \mc{A}$} are distinct, and  assume towards a contradiction that  no set $A \in \mc{A}$ satisfies $\ell_{\infty}^c(A)\subseteq \ker T$. Then, for every $A\in\mathcal{A}$, we can pick a unit vector $x_A\in\ell_{\infty}^c(A)$ such that $Tx_A \neq 0$.
  Set $\mc{A}_{\gamma,n} = \{ A \in \A : \lvert(Tx_A)(\gamma)\rvert\geqslant 1/n\}$ for each $n \in \N$ and $\gamma\in\omega_1$. Since \[ \bigcup_{\gamma\in\omega_1,\,n \in \N} \mc{A}_{\gamma, n} = \{A \in \mc{A} : T x_A \neq 0 \} = \mc{A}, \] which has cardinality~$\omega_2$, we must have $\lvert\mc{A}_{\gamma_0, n_0}\rvert = \omega_2$ for some $n_0 \in \N$ and $\gamma_0\in\omega_1$.
  In fact, it will suffice for our purposes to know that~$\mc{A}_{\gamma_0, n_0}$ contains~$m$ distinct sets $A_1,\ldots,A_m$ for some integer  $m>n_0 \|T\|$.

  For each $i \in \{1, \ldots, m\}$, choose $\sigma_i\in\mathbb{K}$ with $\lvert\sigma_i\rvert = 1$ such that $\sigma_i (T x_{A_i})(\gamma_0) \geqslant 1/n_0$, and set $x = \sum_{i=1}^{m} \sigma_i x_{A_i}\in\ell_\infty^c(\omega_1)$. We shall express this vector as $x=y+z$, 
where $\lVert y\rVert_\infty\leqslant 1$ and $Tz=0$. To this end,  set
  \[ B = \bigcup_{1\leqslant i< j \leqslant m} A_i \cap A_j, \]
which is a finite subset of~$\omega_1$ by the hypothesis on~$\mathcal{A}$. It follows that $z= x\cdot\one_B$ has finite support, so $z\in c_0(\omega_1)\subseteq\ker T$. Furthermore, we have $y = x-z = \sum_{i=1}^{m} \sigma_i x_{A_i}\cdot \one_{A_i\setminus B}$, where the terms  $\sigma_ix_{A_i}\cdot \one_{A_i\setminus B}$ are disjointly supported by the choice of~$B$, and they have norm at most~$1$ by the choices of~$x_{A_i}$ and~$\sigma_i$. This implies that $\lVert y\rVert_\infty\leqslant 1$, and therefore
\[ \lVert T\rVert  \geqslant \lVert Ty\rVert_\infty = \lVert Ty + Tz\rVert_\infty = \lVert Tx\rVert_\infty\geqslant \lvert (Tx)(\gamma_0)\rvert = \Bigl|\sum_{i =1}^m \sigma_i (Tx_{A_i})(\gamma_0) \Bigr|\geqslant \frac{m}{n_0} > \lVert T\rVert, \]
which is obviously absurd. The result follows.
\end{proof}

To prove that it is also consistent with \textsf{ZFC} that the answer to \Cref{question2} is negative, we require a set-theoretic axiom known as~$\clubsuit$, which is independent of \textsf{ZFC}. It is defined as follows (see, \emph{e.g.,} \cite[Definition~4.35]{Hajeketal}). 
\begin{Def}\label{D:clubsuit}
Let~$(\lambda_\alpha)_{\alpha\in\omega_1}$ denote the increasing enumeration of the countably infinite limit ordinals. Then~$\clubsuit$ is the statement  that there exists a transfinite sequence $(B_\alpha)_{\alpha\in\omega_1}$ such that:
\begin{enumerate}[label={\normalfont{(\roman*)}}]
\item\label{clubsuit:item1} for each $\alpha\in\omega_1$, $B_\alpha = \{ \beta_\alpha^{(k)} : k\in\N\}$, where $(\beta_\alpha^{(k)})_{k\in\N}$ is a strictly increasing sequence of ordinal numbers converging to~$\lambda_\alpha;$
\item\label{clubsuit:item2} every uncountable subset of~$\omega_1$ contains~$B_\alpha$ for some $\alpha\in\omega_1$.
\end{enumerate}
\end{Def}

\begin{Thm}\label{t:clubsuit_kernel} Assume~$\clubsuit$. Then a bounded operator $T \colon \ell_{\infty}^c(\omega_1) \to \ell_{\infty}^c(\omega_1)$ exists for which $c_0(\omega_1)\subseteq\ker T$, but  $\ell_\infty^c(A)\nsubseteq\ker T$ for every $A \in [\omega_1]^{\omega_1}$. 
\end{Thm}

\begin{proof} Fix a free ultrafilter~$\EuScript{U}$ on~$\N$, take  a transfinite sequence~$(B_\alpha)_{\alpha\in\omega_1}$ as in \Cref{D:clubsuit},  and let $\alpha\in\omega_1$.  Writing $B_\alpha = \{ \beta_\alpha^{(k)} : k\in\N\}$ as in \Cref{D:clubsuit}\ref{clubsuit:item1}, we can define a  bounded linear functional $\varphi_{\alpha} \colon \ell_{\infty}^c(\omega_1) \to \mathbb{K}$ of norm~$1$ by
  \[ \varphi_\alpha(x) = \lim_{k \to \EuScript{U}} x(\beta_\alpha^{(k)}). \] 
  We note that $\varphi_{\alpha}(\one_{B_{\alpha}}) = 1$ because $\mathbbm{1}_{B_{\alpha}}(\beta_{\alpha}^{(k)}) =1$ for each~$k \in \mathbb{N}$. Furthermore, the fact that $\varphi_{\alpha}(\one_{\{\gamma\}}) = 0$ for each $\gamma\in\omega_1$ implies that
  \[ \ker\varphi_\alpha\supseteq\overline{\operatorname{span}}\,\{\one_{\{\gamma\}} : \gamma\in\omega_1\} = c_0(\omega_1). \]
Hence we can define a bounded operator $T \colon \ell_{\infty}^c(\omega_1) \to \ell_{\infty}(\omega_1)$ of norm~$1$ with $c_0(\omega_1) \subseteq \ker T$ by
  \[ (Tx)(\alpha) = \varphi_{\alpha}(x)\qquad (x\in\ell_\infty^c(\omega_1),\,\alpha\in\omega_1). \]

To see that the range of~$T$ is contained in~$\ell_\infty^c(\omega_1)$, let $x\in\ell_\infty^c(\omega_1)$, and take $\alpha_0\in\omega_1$ such that $\operatorname{supp} x\subseteq [0,\lambda_{\alpha_0}]$. Then, for every countable ordinal number $\alpha>\alpha_0$, we have
\[ \lambda_{\alpha_0} < \lambda_\alpha = \lim_{k\to\infty}\beta_\alpha^{(k)}, \]
so there are at most fi\-nite\-ly many $k\in\N$ such that $\beta_\alpha^{(k)}\le\lambda_{\alpha_0}$, and therefore at most fi\-nite\-ly many  $k\in\N$ such that $x(\beta_\alpha^{(k)})\ne 0$. Hence $\varphi_\alpha(x) =0$ whenever $\alpha>\alpha_0$.  This shows that $\operatorname{supp} Tx\subseteq [0,\alpha_0]$, and consequently $Tx\in\ell_\infty^c(\omega_1)$.
  
Finally,  we verify that    $\ell_\infty^c(A)\nsubseteq\ker T$ for every $A\in[\omega_1]^{\omega_1}$.   \Cref{D:clubsuit}\ref{clubsuit:item2} implies that we can find $\alpha\in\omega_1$ such that $B_\alpha\subseteq A$. Then $\one_{B_\alpha}\in \ell_\infty^c(A)$ because $B_\alpha$ is countable, and we have $(T\one_{B_\alpha})(\alpha) = \varphi_{\alpha}(\one_{B_{\alpha}})= 1$, so  $\one_{B_\alpha}\notin\ker T$. The conclusion follows.
\end{proof}

Combining Theorems~\ref{t:gch_kernel} and~\ref{t:clubsuit_kernel}, we see that the statement~\eqref{ZFC:indep} is independent of \textsf{ZFC}, as previously claimed.\smallskip

This raises the question: within \textsf{ZFC}, what can we say about kernels of bounded operators on~$\ell_\infty^c(\Gamma)$? To address it, we require the following piece of terminology: a bounded operator $T\colon X\to Y$ between Banach spaces~$X$ and~$Y$ \emph{fixes a copy} of a  Banach space~$Z$ if there is a bounded operator $U\colon Z\to X$ such that $TU$ is an isomorphic embedding. 
\begin{Prop}\label{p:countably_supp_ideal}  
  Let~$\Gamma$ be a cardinal number with uncountable cofinality, let~$X$ be a Banach space, and suppose that $T\colon\ell_{\infty}^{c}(\Gamma)\to X$ is a bounded operator which does not fix a copy of~$c_0(\Gamma)$. Then $\ell_{\infty}^c(A)\subseteq\ker T$ for some $A \in [\Gamma]^{\Gamma}$.
\end{Prop}

\begin{proof}
  We shall  prove the contrapositive statement, so  suppose that $\ell_{\infty}^c(A) \nsubseteq \ker T$ for every \mbox{$A \in [\Gamma]^{\Gamma}$}.
  Since~$\Gamma$ is infinite, we can take a family $(A_{\gamma})_{\gamma\in\Gamma}$ of mutually disjoint sets in~$[\Gamma]^\Gamma$.   
The hypothesis implies that we can choose a unit vector $x_{\gamma}\in \ell_{\infty}^{c}(A_\gamma)\setminus\ker T$ for each $\gamma\in\Gamma$.   Then  the transfinite sequence $(x_{\gamma})_{\gamma\in\Gamma}$ consists of disjointly supported unit vectors in~$\ell_{\infty}^{c}(\Gamma)$, so we can define a linear isometry $S\colon c_0(\Gamma)\to\ell_\infty^c(\Gamma)$ by setting $Se_{\gamma} = x_\gamma$ for each $\gamma\in\Gamma$. 

Combining the fact that~$\Gamma$ has uncountable cofinality with the observation that
\[ \Gamma = \bigcup_{n\in\N} \Gamma_n,\quad\text{where}\quad \Gamma_n =  \Bigl\{\gamma\in\Gamma : \|Tx_{\gamma} \|_{\infty} \geqslant \frac1n\Bigr \}, \]
we conclude that $\lvert\Gamma_n\rvert = \Gamma$ for some $n\in\N$. Then
$\inf_{\gamma \in\Gamma_n}\lVert TSe_{\gamma} \rVert_{\infty}\ge 1/n>0$, so a famous result of Rosenthal stated in \cite[Remark~1 (page~30)]{Ros} implies that we can find $\Gamma'\in[\Gamma_n]^\Gamma$ such that the restriction of~$TS$ to $c_0(\Gamma')\subseteq c_0(\Gamma)$ is an isomorphic embedding. This shows that~$T$ fixes a copy of~$c_0(\Gamma)$.
\end{proof}

\Cref{p:countably_supp_ideal} should be compared with the  following result of W\'{o}jtowicz \cite[Theorem~1]{Woj}.

\begin{Thm}[W\'{o}jtowicz]
  Let $T\colon\ell_{\infty}(\Gamma)/c_0(\Gamma)\to X$ be a bounded operator, where~$\Gamma$ is infinite and~$X$ is a Banach space. If $T$ does not fix a copy of~$\ell_{\infty}$, then~$\ker T$ contains a closed subspace which is isomorphic to~$c_0$.
\end{Thm}

\begin{Rem}\label{R:countably_supp_ideal}
  Let $X = \ell_{\infty}^c(\Gamma)$, and write $\mathscr{B}(X)$ for the Banach algebra of bounded opera\-tors~$X\to X$. Johnson, Kania and Schechtman \cite[Theorems~1.1 and~3.14]{JKSch}  have shown that the set $\mathscr{S}_{X}(X) = \{ T\in\mathscr{B}(X) : T\ \text{does not fix a copy of}\ X\}$ is the unique maximal ideal of~$\mathscr{B}(X)$ and that $T\in\mathscr{S}_X(X)$ if and only if $T$ does not fix a copy of~$c_0(\Gamma)$. Suppose that ~$\Gamma$ has uncountable cofinality. Then we can combine this result with \Cref{p:countably_supp_ideal} to deduce that for every  $T\in\mathscr{S}_X(X)$,   there exists $A \in [\Gamma]^{\Gamma}$ such that $\ell_{\infty}^c(A) \subseteq \ker T$.
\end{Rem}
\appendix
\section{The proof of Lemma~\ref{CombLemma}}\label{app:A}
\begin{proof}[Proof of Lemma~{\normalfont{\ref{CombLemma}}}]
  We begin by recalling the statement of the lemma and introduce some additional notation that will help us present its proof concisely. The set-up is that  $\sigma_1,\ldots,\sigma_m$ are scalars, while $A_1,\ldots,A_m$ are  distinct subsets  of~$\Gamma$ for which we define
  \[ A^d_j = A_j\setminus\underset{k\ne j}{\bigcup_{k=1}^m}A_k\qquad (j\in\{1,\ldots,m\}).  \]
  Then the claim is that   the functions $\lhs,\rhs\colon \Gamma\to\mathbb{K}$ given by 
    \[ \lhs= \sum_{j=1}^m\sigma_j \one_{A^d_j}\qquad\text{and}\qquad \rhs =  \sum_{\emptyset\ne N\subseteq\{1,\ldots,m\}} (-1)^{\lvert N\rvert+1}\Bigl(\sum_{j\in N}\sigma_j\Bigr)\one_{\bigcap_{j\in N} A_j} \]
are equal.

    We verify this equality pointwise. Take $\gamma\in\Gamma$, define \mbox{$M_\gamma=\bigl\{j\in\{1,\ldots,m\} : \gamma\in A_j\bigr\}$}, and observe that $\gamma\in\bigcap_{j\in N} A_j$ if and only if $N\subseteq M_\gamma$. 
Clearly  $\lhs(\gamma) = 0 = \rhs(\gamma)$ if $M_\gamma=\emptyset$, while $\lhs(\gamma)=\sigma_j=\rhs(\gamma)$ if $M_\gamma=\{j\}$ for some $j\in\{1,\ldots,m\}$. Hence we may suppose that $M_\gamma$ has at least two elements. In this case $\gamma\in \bigcup_{k=1,\, k \neq j}^mA_k$ for every $j\in\{1,\ldots, m\}$, so $\lhs(\gamma)=0$. On the other hand, using the observation above, we find
\[ \rhs(\gamma) = \!\!\sum_{\emptyset\ne N\subseteq M_\gamma}\!\! (-1)^{\lvert N\rvert+1}\Bigl(\sum_{j\in N}\sigma_j\Bigr) =  \sum_{k=1}^{\lvert M_\gamma\rvert} (-1)^{k+1} \sum_{\substack{N\subseteq M_\gamma\\ \lvert N\rvert = k}} \Bigl(\sum_{j\in N}\sigma_j\Bigr) 
= \sum_{k=1}^{\lvert M_\gamma\rvert} (-1)^{k+1} \sum_{j\in M_\gamma} a_{k,j} \sigma_j, \]
where $a_{k,j}$ denotes the number of subsets~$N$ of~$M_\gamma$ of size~$k$ such that~$j\in N$. Since $a_{k,j} = \binom{\lvert M_\gamma\rvert-1}{k-1}$, by re-indexing and applying the binomial formula, we conclude that
\[  \rhs(\gamma) = 
  \sum_{\ell=0}^{\lvert M_\gamma\rvert-1}(-1)^{\ell}\binom{\lvert M_\gamma\rvert-1}{\ell}\sum_{j\in M_\gamma}\sigma_j = (-1+1)^{\lvert M_\gamma\rvert-1}\sum_{j\in M_\gamma}\sigma_j = 0, \]
as desired.
\end{proof}

\begin{Rem}
The aim of this remark is to explain why we may view \Cref{CombLemma}  as a weighted variant of the Inclusion--Exclusion Principle. 
Let~$\mu$ be  a scalar-valued, finitely additive measure defined on a Boolean subalgebra~$\mathcal{F}$ of the power set of~$\Gamma$. 
Then, taking $\sigma_1=\cdots =\sigma_m =1$ in \Cref{CombLemma} and  integrating both sides of~\eqref{CombLemma:eq2} with respect to~$\mu$,  we obtain
\begin{equation*}
\sum_{j=1}^m \mu (A_j^d) = \sum_{\emptyset\ne N\subseteq\{1,\ldots,m\}} (-1)^{\lvert N\rvert+1} \cdot \lvert N\rvert \cdot  \mu\Bigl(\bigcap_{j\in N} A_j \Bigr)\qquad (m\in\mathbb{N},\,A_1,\ldots,A_m \in\mathcal{F}). \end{equation*} 
\end{Rem}

\begin{Ack}
  The first-named author is grateful to Max Arnott for bringing this problem to his attention. He acknowledges with thanks the financial support provided by the Jagiellonian Scholarship ``Initiative of Excellence --- Jagiellonian University'' awarded by the Faculty of Mathematics and Computer Science. We would also like to thank Yemon Choi and Tommaso Russo for some enlightening conversations, and Jes\'{u}s Castillo for mentioning the original reference~\cite{Ka} to us. Finally, we thank the referee for their careful reading of our paper and helpful suggestions.
\end{Ack}

%%%%%%%%%%%%%%%%%%%%%%%%%%%%%%%%%%%%%%%%%%%%%%%%%%%%%%%%%%%%%%

\bibliographystyle{amsplain}

\begin{thebibliography}{10}
\bibitem{alexandroff} P.~Alexandroff and P.~Urysohn, \emph{Memoire sur les espaces topologiques compacts dedie a
Monsieur D.~Egoroff}, 1929.

\bibitem{AL} M.~Arnott and N.~J.~Laustsen, Kernels of bounded operators on the classical transfinite Banach sequence spaces, preprint, 2023.  

\bibitem{Avilesetal} A.~Avil{\'e}s, F.~Cabello S{\'a}nchez, J.~M.~F.~Castillo, M.~Gonz{\'a}lez and Y.~Moreno, \emph{Separably injective Banach spaces}, Lect.\ Notes Math., Springer-Verlag, 2016.	

\bibitem{Baum} J.~E.~Baumgartner, Almost-disjoint sets, the dense set problem and the partition calculus, \emph{Ann.\ Math.\ Logic}~\textbf{10} (1976), 401--439.	
   
\bibitem{Hajeketal} P.~H\'ajek, V.~Montesinos, J.~Vanderwerff and V.~Zizler, \emph{Biorthogonal systems in Banach spaces}, CMS Books in Mathematics, Springer-Verlag, 2008.

\bibitem{JKSch} W.~B.~Johnson, T.~Kania and G.~Schechtman, Closed ideals of operators on and complemented subspaces of Banach spaces of functions with countable support, \emph{Proc.\ Amer.\ Math.\ Soc.}~\textbf{144} (2016), 4471--4485.	

\bibitem{JL} W.~B.~Johnson and J.~Lindenstrauss, Some remarks on weakly compactly generated Banach spaces, \emph{Israel J.\ Math.}~\textbf{17} (1974), 219--230.
  
\bibitem{Ka} N.~J.~Kalton, Spaces of compact operators, \emph{Math.\ Ann.}~\textbf{208} (1974), 267--278. 
	
\bibitem{LaW} N.~J.~Laustsen and J.~T.~White, Subspaces that can and cannot be the kernel of a bounded operator on a Banach space, in \emph{Banach algebras and applications. Proceedings of the 23rd international conference, University of Oulu, Finland, July 3--11, 2017} (ed.~M.~Filali), \emph{De Gruyter Proceedings in Mathematics} (2020), 189--196.
  
\bibitem{lt} J.~Lindenstrauss and L.~Tzafriri, \emph{Classical Banach spaces~I}, Springer-Verlag, 1977.	

\bibitem{MT} C.~Mariconda and A.~Tonolo, \emph{Discrete Calculus: Methods for Counting}, Springer-Verlag, 2016.	

\bibitem{PS} A.~Pe\l{}czy\'{n}ski and V.~N.~Sudakov, Remark on non-complemented subspaces of the space~$m(S)$, \emph{Colloq.\ Math.}~\textbf{9} (1962), 85--88.

\bibitem{Ros} H.~Rosenthal, On relatively disjoint families of measures, with some applications to Banach space theory, \emph{Studia\ Math.}~\textbf{37} (1970), 13--36.

\bibitem{Sier} W.~Sierpi\'{n}ski, Sur une d\'{e}composition d'ensembles, \emph{Monatsh.\ Math.\ Phys.}~\textbf{35} (1928), 239--242.
  
\bibitem{Wark} H.~M.~Wark, A non-separable reflexive Banach space on which there are few operators, \emph{J.\ London Math.\ Soc.}~\textbf{64} (2001), 675--689.

\bibitem{Whit} R.~Whitley, Projecting~$m$ onto~$c_0$, \emph{Amer.\ Math.\ Monthly}~\textbf{73} (1966), 285--286. 
  
\bibitem{Woj} M.~W\'{o}jtowicz, Operators preserving $\ell_{\infty}$, \emph{RACSAM}~\textbf{108} (2014), 511--517.
\end{thebibliography}

\end{document}